\documentclass{amsart}

\usepackage{amsmath,amssymb,amsthm}
\usepackage{tznotes}
\usepackage{thomas}
\usepackage{psfrag}
\usepackage{graphicx}
\usepackage{url}
\usepackage{color}
\usepackage{ifpdf}
\newcommand{\warning}[1]{\textbf{#1}}
\DeclareMathOperator{\card}{card}

\usepackage{graphicx}
\ifpdf
\fi

\newcommand{\ep}{\epsilon}
\newcommand{\del}{\delta}

\newcommand{\reals}{\mathbb{R}}

\newcommand{\nats}{\mathbb{N}}

\newcommand{\Heis}{\mathbb{H}}

\newcommand{\nbhd}{\mathcal{N}}

\newcommand{\inv}{^{-1}}

\newcommand{\ovl}{\overline}

\newcommand{\into}{\hookrightarrow}

\newcommand{\Hdim}{\mathcal{H}}
\newcommand{\subeq}{\subseteq}

\newcommand{\bslash}{\backslash}

\newcommand{\iQ}{\mathcal{Q}}

\newcommand{\eL}{\textrm{L}}

\def\diam{\operatorname{diam}}
\def\card{\operatorname{card}}
\def\dist{\operatorname{dist}}

\def\modu{\operatorname{\text{mod}}}
\def\length{\operatorname{length}}

\def\Lip{\operatorname{Lip}}
\def\lip{\operatorname{lip}}
\def\loc{\operatorname{loc}}
\newcommand{\Lloc}{L_{\loc}}

\def\Xint#1{\mathchoice
{\XXint\displaystyle\textstyle{#1}}%
{\XXint\textstyle\scriptstyle{#1}}%
{\XXint\scriptstyle\scriptscriptstyle{#1}}%
{\XXint\scriptscriptstyle\scriptscriptstyle{#1}}%
\!\int}
\def\XXint#1#2#3{{\setbox0=\hbox{$#1{#2#3}{\int}$}
\vcenter{\hbox{$#2#3$}}\kern-.5\wd0}}

\def\mint{\Xint-}

\begin{document}
\author{K. Wildrick}
\address{Mathematisches Institut, Universit\"at Bern, Sidlerstrasse 5, 3012 Bern, Switzerland}
\thanks{K.~W. was supported by Academy of Finland grant 128144, the Swiss National Science Foundation, European Research Council Project CG-DICE, and the European Science Council Project HCAA}
\email{kevin.wildrick@math.unibe.ch}
\author{T. Z\"urcher}
\address{Matematiikan ja tilastotieteen laitos, Jyv\"askyl\"an Yliopisto, PL 35 (MaD)
40014 Jyv\"askyl\"a, Finland}
\thanks{T.~Z. was supported by the Swiss National Science Foundation grant PBBEP3\_130157 and by the Academy of Finland grant number 251650}
\email{thomas.t.zurcher@jyu.fi}

\title[Sharp differentiability results and non-embedding]{Sharp differentiability results for $\lip$ and applications to non-embedding}
\begin{abstract}We give a sharp condition on the lower local Lipschitz constant of a mapping from a metric space supporting a Poincar\'e inequality to a Banach space with the Radon-Nikodym property that guarantees differentiability at almost every point. We apply these results to obtain a non-embedding theorem for a corresponding class of mappings. \end{abstract}
\maketitle

\section{Introduction}
In 1919, Rademacher proved that
Lipschitz mappings between Euclidean spaces are differentiable
almost everywhere, \cite{Rademacher}. As the Lipschitz condition is global, while differentiability is local, Stepanov considered the set
\begin{equation*}
S(f):=\{x\in \field{R}^n:\, \Lip f(x) <\infty\},
\end{equation*}
where
\begin{equation*}
\Lip f(x):=\limsup_{r\to 0} \frac{\sup_{y\in
B(x,r)}\abs{f(y)-f(x)}}{r}
\end{equation*}
and proved the following generalization of Rademacher's theorem: a function
$f\colon \field{R}^n\to \field{R}$ is differentiable almost
everywhere in $S(f)$ \cite{Stepanov}. A second strengthening of Rademacher's theorem states that if $p>n$, then Sobolev functions in $W^{1,p}_{\loc}(\reals^n)$ are
differentiable almost everywhere; this is due to Cesari \cite{Cesari} when $n=2$.  Calder\'on \cite{Calderon} and later Stein \cite{Stein} generalized and sharpened this result: if the weak gradient of a function on $\reals^n$ is in the Lorentz space $L^{n,1}$, then the function is differentiable almost everywhere. Moreover,  $L^{n,1}$ is the largest of the Lorentz spaces $\{L^{n,q}: 1\leq q \leq n\}$ to have
this property. In fact, Calder\'on considered Orlicz spaces; the relation between these spaces and the Lorentz spaces is clarified in \cite{KKM} and \cite{FineBehavior}. Heuristically, a function on $\reals^n$ with a weak gradient in $L^{n,1}$ shares properties with absolutely continuous functions of a single variable.  This principle has been extended to include various higher dimensional notions of absolute continuity \cite{mal_absolutely_1999}, \cite{KKM}, and to apply in more general settings \cite{romanov_absolute_2008}, \cite{Ranjbar}, \cite{SpaceFilling}.

Yet another generalization of Rademacher's theorem was explored by Balogh and Cs\"ornyei, who considered
\begin{equation*}
\lip f(x):=\liminf_{r\to 0} \frac{\sup_{y\in
B(x,r)}\abs{f(y)-f(x)}}{r}
\end{equation*}
instead of $\Lip f$ \cite{BC}. They observed that
Stepanov's theorem with $\Lip f$ replaced by $\lip f$ does not hold
in general, and provided two examples highlighting the obstructions. The first showed that control on the integrability of $\lip f$ is needed, and the second
showed that an upper bound on the size of
the set where $\lip f$ is infinite is required.  On the other hand, they gave the following positive result, which, when combined with the result of Cesari, can be consider to be of ``Stepanov-type".

\begin{theorem}[Balogh-Cs\"ornyei]\label{BC theorem} Let $\Omega \subeq \reals^n$ be a domain, and let $f \colon \Omega \to \reals$ be a continuous function. Assume that $\lip f(x) < \infty$ for $x \in \Omega\bslash E$, where the exceptional set $E$ has $\sigma$-finite $(n-1)$-dimensional Hausdorff measure, and that $\lip f \in \Lloc^p(\Omega)$ for some $1 \leq p \leq \infty$.  Then $f \in W^{1,p}_{\loc}(\Omega)$.
\end{theorem}

In the foundational work \cite{Cheeger99}, Cheeger generalized Rademacher's theorem to the setting of Lipschitz functions on metric measure spaces that support a Poincar\'e inequality using the notion of a \emph{measurable differentiable structure}. Other generalizations in this setting followed \cite{BRZ07}, \cite{Z}, \cite{RNP}.


The first result of this paper sharpens the theorem of Balogh and Cs\"ornyei to the Lorentz scale, in the spirit of Stein. Our result is in the setting of mappings between metric measure spaces, although it is new even for functions on Euclidean space. The assumptions are standard for this setting, and can be heuristically understood to mean that the domain space $(X,d,\mu)$ is of dimension $Q$ and contains a large enough collection of rectifiable curves to support ``first order calculus". For more details on these definitions, as well as the usual generalization of $\Lip$ and $\lip$ to the setting of mappings between metric spaces, see Section~\ref{notation} below. 

\begin{theorem}\label{main} Let $Q \geq 1$ and $1 \leq q \leq Q$.  Let $(X,d,\mu)$ be a complete and Ahlfors $Q$-regular metric space that supports a $q$-Poincar\'e inequality, and let $Y$ be any metric space. Consider a continuous mapping $f \colon X \to Y$, and set 
\begin{equation*}
E =\{x \in X: \lip f(x)=\infty\}.
\end{equation*}
If $\lip f \in L^{Q,1}(X)$ and either
\begin{itemize}
\item $q=1$ and $E$ has $\sigma$\nobreakdash-finite $(Q-1)$\nobreakdash-dimensional Hausdorff measure, or
\item $q>1$ and $E$ has Hausdorff dimension at most $(Q-q)$,
\end{itemize}
then $f$ has an upper gradient in the Lorentz space $L^{Q,1}(X).$
\end{theorem}

As in the Euclidean setting, the assumption that an upper gradient of the mapping $f$ is in the Lorentz space $L^{Q,1}(X)$ implies that $f$ shares many properties with an absolutely continuous function on $\reals$; see Theorem~6.3, Corollary~6.7, and Lemma~6.8 in \cite{SpaceFilling}. We record some of these properties in the following statement. 

\begin{corollary}\label{condition n} Assume the notation and hypotheses of Theorem \ref{main}. If $N \subeq X$ is a set of zero $Q$-dimensional Hausdorff measure, then so is $f(N)$. Moreover, the image $f(X)$ is the countable union of Lipschitz images of $X$ and a set of zero $Q$\nobreakdash-dimensional Hausdorff measure. In particular, if $X = \reals^n$, then $f(X)$ is a countably $n$-rectifiable subset of $Y$.
\end{corollary}

Hanson has recently given an example showing the sharpness of Theorem \ref{main} with respect to the size of the exceptional set \cite[Theorem~2.3]{Hanson}.  Here we provide an example that shows sharpness with respect to the integrability of $\lip$. Our construction draws on ideas of Hanson, Mal\'y \cite{maly}, and our previous work on the capacity of points \cite{SpaceFilling}. It is substantially different from the corresponding example in \cite{BC}.

\begin{theorem}\label{example1} Let $n \geq 2$ be an integer, and let $\mathcal{S}$ be a rearrangement invariant Banach function space containing a compactly supported function $g \notin L^{n,1}(\reals^n)$.  Then there exists a compactly supported continuous function $f \colon \reals^n \to \reals$ with the following properties:
\begin{itemize}
\item $\lip f(x) < \infty$ for all $x \in \reals^n$,
\item $\lip f(x) \in \mathcal{S}$, 
\item the set of points at which $f$ fails to be differentiable has positive $n$-dimensional Hausdorff measure.
\end{itemize}
\end{theorem}

A reader unfamiliar with the notion of a rearrangement invariant Banach function space may consult \cite{BS}. As an example, we point out that for any $1<q\leq n$, the Lorentz space $L^{n,q}(\reals^n)$ satisfies the hypotheses of the above theorem.

In the case that the target of the mapping in Theorem \ref{main} is a Banach space with the Radon-Nikodym property, we provide in Theorem \ref{Stepanov} below a differentiation result in the spirit of Stein. Here differentiability is meant in the sense of Cheeger, i.e., in terms of a measurable differentiable structure $\{(X_\alpha,\varphi_\alpha)\}$ on a metric measure space $(X,d,\mu)$. Each \emph{chart} $(X_\alpha,\varphi_\alpha)$ consists of a measurable subset $X_{\alpha}$ and a Lipschitz function $\varphi_{\alpha} \colon X \to \reals^{N(\alpha)}$, where $N(\alpha) \in \nats$ is called the \emph{dimension} of the chart. The collection of charts is assumed to be countable and to cover $X$ upto a set of $\mu$-measure $0$. Roughly speaking, the differentiability of a mapping from $f \colon X \to V$, where $V$ is a Banach space, with respect to the chart $(X_\alpha, \varphi_\alpha)$ means existence of a measurable function $g \colon X_\alpha \to V$ so that 
$\left\langle \varphi, g \right\rangle$ approximates $f$ to first order almost everywhere on $X_\alpha$. Built-in to the definition of a measurable differentiable structure is a Rademacher theorem: every real-valued Lipschitz function is differentiable with respect to each chart. It is a deep result of Cheeger that such a measurable differentiable structure exists whenever $(X,d,\mu)$ is doubling and supports a Poincar\'e inequality \cite{Cheeger99}; one could call this result the Cheeger-Rademacher Theorem.  This was extended by Cheeger and Kleiner in \cite{RNP} to include the differentiability of  Lipschitz mappings to any Banach space that has the Radon-Nikodym property. See Section \ref{diff Section} below for more details and definitions.

\begin{theorem}\label{Stein}  Let $Q \geq 1$.  Let $(X,d,\mu)$ be a complete and Ahlfors $Q$-regular metric space that supports a $Q$-Poincar\'e inequality, and let $V$ be a Banach space with the Radon-Nikodym property. Let $f \colon X \to V$ be a continuous mapping with an upper gradient in the Lorentz space $L^{Q,1}(X)$.  Then $f$ is almost everywhere differentiable with respect to any measurable differentiable structure on $X$.
\end{theorem}

Theorem \ref{Stein}, when combined with Theorem \ref{main}, provides a ``Stepanov-Stein" result for $\lip$. 

As discussed in \cite{RNP}, differentiability theorems in the setting of metric measure spaces lead to corresponding non-embedding theorems. In certain cases, these non-embedding results have important applications in theoretical computer science \cite{Compression}.  

The differentiability of Lipschitz mappings into a Banach space with the Radon-Nikodym property leads to a bi-Lipschitz non-embedding theorem. Our differentiation result Theorem \ref{Stein} leads to a non-embedding theorem for a larger class of mappings. 

For a mapping $f\colon X \to Y$ of metric spaces, we denote
$$\mathcal{S}_{f} = \{x \in X: \Lip f(x) < \infty\}.$$
We say that $f$ is an \emph{embedding} if it is a homeomorphism of $X$ onto $f(X)$.

\begin{theorem}\label{RNP non}  Let $Q \geq 1$.  Let $(X,d,\mu)$ be a complete and Ahlfors $Q$-regular metric space that supports a $Q$-Poincar\'e inequality, and let $V$ be a Banach space with the Radon-Nikodym property. Let $\iota \colon X \to V$ be an embedding with an upper gradient in the Lorentz space $\eL^{Q,1}(X)$. If 
\begin{itemize}
\item[(i)]$\Hdim^Q_V(\iota(X) \bslash \mathcal{S}_{\iota\inv}) = 0$,
\item[(ii)]$\Hdim^Q_X(X \bslash \iota\inv(\mathcal{S}_{\iota\inv}))=0$, 
\end{itemize}
then for any coordinate chart $(X_\alpha,\varphi_\alpha)$ of any measurable differentiable structure on $X$, it holds that
$Q \leq N(\alpha).$
\end{theorem}
Combining Theorem \ref{main} with Theorem \ref{RNP non} produces a non-embedding result for $\lip$. 

We note that in the setting of Theorem \ref{RNP non}, the existence of an upper gradient of $\iota$ in $\eL^{Q,1}(X)$ implies that $\Lip \iota(x)<\infty$ at $\Hdim^Q$-almost every point, and moreover it implies that $\iota$ maps $\Hdim_X^Q$-null sets to $\Hdim_V^Q$-null sets \cite[Section~6]{SpaceFilling}. Hence, if it is also known that the image of $\iota$ is an Ahlfors $Q$-regular metric space supporting a $Q$-Poincar\'e inequality, then conditions (i) and (ii) above could be replaced by the symmetric requirement that $\iota\inv$ have an upper gradient in $\eL^{Q,1}(\iota(X)).$ 

Let us consider the concrete example of the Heisenberg group $\Heis$, which is an Ahlfors $4$-regular metric measure space, is  homeomorphic to $\reals^3$, supports a $1$\nobreakdash-Poincar\'e inequality, and has a measurable differentiable structure in which each chart has dimension $2$ \cite{primer}, \cite{PansuAnn}. Hence, in the language of Theorem~\ref{RNP non}, $Q>N(\alpha)$ for each $\alpha$. The differentiability theorem for Lipschitz mappings into a Banach space $V$ with the Radon-Nikodym property given in \cite{RNP} allows one to prove that there is no bi-Lipschitz embedding $\iota \colon \Heis \to V$. Our result states that there is not even an embedding $\iota \colon \Heis \to V$ satisfying the hypotheses of Theorem~\ref{RNP non}. We point out that the identity mapping from $\Heis$ to $\reals^3$ is Lipschitz on compact sets, and so some metric condition on the inverse of the embedding is needed.



The standard assumptions of Ahlfors regularity and a Poincar\'e inequality in the above theorems place conditions on large scales that are not natural for the conclusions, which are local in nature. However, the theorems still hold for arbitrary domains in Euclidean space, as can be seen by restriction to a sufficiently small closed ball. Moreover, it appears that Theorems \ref{main}, \ref{Stein}, and \ref{RNP non} remain valid when small scale and localized versions of these assumptions, as well as a local Lorentz integrability condition, are used instead. We have used global assumptions only for clarity of exposition, and we leave the aforementioned generalizations, which we expect can be proven in the same fashion, to the reader; see \cite[Remark~6.11]{SpaceFilling}.

This paper is organized as follows. In Section~\ref{notation} we establish notation and definitions regarding metric measure spaces, Lorentz integrability, and measurable differentiable structures. We then prove Theorem \ref{main} in Section \ref{main proof section}. This is followed in Section~\ref{Stein section} by a proof of the differentiation result Theorem \ref{Stein}.  Section \ref{weak tan section} discusses the basic properties of weak tangents and weak tangent mappings at points of differentiability, which leads to a proof of the non-embedding result Theorem~\ref{RNP non}. Finally, we give the example described in Theorem ~\ref{example1} in Section \ref{example section}. 

We wish to thank Jeff Cheeger, Amiran Gogatishvilli, Bruce Hanson, and Pekka Koskela for helpful comments.

\section{Notation and basic definitions}\label{notation}

\subsection{Metric measure spaces}
Given a metric space $(X,d)$, we denote the open ball centered at a point $x \in X$ of radius $r>0$ by
$$B_X(x,r)=\{y \in X: d(x,y)<r\}$$ 
and the corresponding closed ball by
$$\ovl{B}_X(x,r) = \{y \in X: d(x,y) \leq r\}.$$
When there is no danger of confusion, we often write $B(x,r)$ in place of $B_X(x,r)$.  A similar convention will be used for all objects that depend implicitly on the ambient space.  Given a subset $A$ of $X$ and a number $\ep>0$, we notate the $\ep$-neighborhood of $A$ by
$$\nbhd(A,\ep) = \{x \in X : \dist(A,x) < \ep\}.$$
Given an open ball $B=B(x,r)$ and a parameter $\lambda > 0$, we set $\lambda B =B(x,\lambda r)$.

A \emph{metric measure space} is a triple $(X,d,\mu)$ where $(X,d)$ is a metric space and $\mu$ is a measure on $X$. For our purposes, a \emph{measure} is a nonnegative countably subadditive set
function defined on all subsets of a measure space that gives the value $0$ to the empty set. We further
assume that measures are Borel inner and outer regular, and the collection of measurable sets is given by the completion of the Borel $\sigma$-algebra. We will often suppress the reference to the metric $d$ and the measure $\mu$ when they are understood.

The metric measure space $(X,d,\mu)$ is \emph{doubling} if balls have finite and
positive measure, and there is a constant $C\geq 1$ such
$\mu(2B)\leq C \mu(B)$ for any open ball $B$ in $X$. If $(X,d,\mu)$ is a doubling metric measure space, then the metric space $(X,d)$ enjoys the following property, also called \emph{doubling}: there is a number $n \in \nats$ such that any ball in $X$ of radius $r>0$ can be covered by at most $n$ balls of radius $r/2$.  It is easy to see that a doubling metric space is complete if and only if it is proper, i.e., closed and bounded sets are compact. Moreover, doubling metric spaces are separable.

Doubling metric spaces are precisely those of finite Assouad dimension \cite[Chapter 10]{LAMS}. However, this notion of dimension is not uniform; a doubling metric space may have some parts or scales where the space appears to be of lower dimension than is actually the case. We will have occasion to be more precise: the metric measure space $(X,d, \mu)$ is called \textit{Ahlfors $Q$-regular} if there exists a constant $C\geq 1$ such that for each point $a \in X$ and each radius $0<r<2\diam X$,
\begin{equation}\label{strong reg def} \frac{r^Q}{C} \leq \mu(B(a,r)) \leq Cr^Q.\end{equation}
Note that an Ahlfors $Q$-regular space is doubling, quantitatively.

We denote the $Q$-dimensional Hausdorff measure on a metric space $(X,d)$ by $\Hdim^Q_X$. If a metric measure space $(X,d,\mu)$ is Ahlfors $Q$-regular, then so is the metric measure space $(X,d,\Hdim^Q_X)$, and so we will often work directly with the Hausdorff measure in this case.

\subsection{Modulus, upper gradients, and Poincar\'e inequalities for metric space valued mappings}

A key idea in theory of analysis on metric spaces is to measure the plenitude of curves in a given space. For $p \geq 0$, the \emph{$p$-modulus} of a family $\Gamma = \{\gamma\colon [0,1] \to X\}$ of rectifiable paths in a metric measure space $X$ is given by
$$\modu_p(\Gamma)=\inf \int_X g^p d\mu,$$
where the infimum is taken over all \emph{admissible metrics} for $\Gamma$, namely, over all Borel functions $g\colon X \to [0,\infty]$ such that for all $\gamma \in \Gamma$
$$\int_\gamma g \ ds \geq 1.$$

The notion of modulus is closely tied to an analogue of the norm of a gradient of a  function on Euclidean space. Let $f \colon X \to Y$ be a mapping between metric spaces.   An \textit{upper gradient} of $f$ is a Borel function $g \colon X \to [0,\infty]$ such that for each rectifiable path $\gamma \colon [0,1] \to X$,
\begin{equation}\label{ug ineq}d_Y(f(\gamma(0)),f(\gamma(1))) \leq \int_{\gamma} g \ ds.\end{equation}
If it is only known that the set of rectifiable paths not satisfying inequality \eqref{ug ineq} is contained in a path family of zero $p$-modulus, then we say that $g$ is a \emph{$p$-weak upper gradient} of $f$.

If $f$ is locally Lipschitz, then the \emph{upper local Lipschitz constant} of $f$, defined by
$$\Lip(f)(x)=\limsup_{r \to 0} \sup_{y \in B(x,r)} \frac{d_Y(f(x),f(y))}{r},$$
is an upper gradient of $f$ \cite[Proposition 1.11]{Cheeger99}. In fact, in this case, a slightly more delicate argument shows that the \emph{lower local Lipschitz constant} of $f$, defined by
$$\lip(f)(x)=\liminf_{r \to 0} \sup_{y \in B(x,r)} \frac{d_Y(f(x),f(y))}{r},$$
is an upper gradient of $f$.  Both statements may fail if $f$ is not assumed to be locally Lipschitz.

We next briefly discuss integration of metric valued mappings. See \cite[Section 2]{HKST} and \cite[Section 3.3]{SpaceFilling} for a more detailed account. Let $Y$ be any metric space. A mapping $f \colon X \to Y$ is said to be \textit{Bochner measurable} if it is measurable in the usual sense and \textit{essentially separably valued}, meaning that there is a set $N \subeq X$ of measure $0$ such that $f(X\bslash N)$ is a separable subset of $Y$.

The following notion of local integrability of metric space valued mappings is perhaps not yet standard.

\begin{definition}[locally integrable]  \label{L1loc def} A mapping $f\colon X \to Y$ is in the class $\Lloc^1(X;Y)$, i.e., it is said to be \textit{locally integrable}, if it is Bochner measurable and there exists a point $z \in Y$ such that the function $x \mapsto d_Y(f(x),z)$ is in the space $\Lloc^1(X)$.
\end{definition}

In the case that $Y$ is a Banach space, Definition \ref{L1loc def} is equivalent to the condition that the mapping in question is \emph{locally Bochner integrable} (again, see \cite[Section 2]{HKST}). Moreover, a mapping $f \colon X \to Y$ is locally integrable if and only for every isometric embedding $\iota$ of $Y$ into any Banach space $V$, the composition $\iota \circ f \colon X \to V$ is locally Bochner integrable. If $X$ is a separable metric space, any continuous mapping from $X$ to another metric spaces has separable image and is locally Bochner integrable.

Given a measurable subset $E$ of $X$ of finite and positive measure, a Banach space $V$, and a locally Bochner integrable mapping $f \colon X \to V$, we define the \emph{average value of f on $E$} by
$$f_E = \mint_E f \ d\mu = \frac{1}{\mu(E)}\int_E f \ d\mu.$$

Fundamental work of Heinonen and Koskela has resulted in an analytic condition which guarantees the presence of ``many" rectifiable curves in a metric space \cite{Acta}. Fix $p \geq 1$. Let $V$ be a Banach space, let $f\colon X \to V$ be a locally integrable mapping, and let $g\colon X \to [0,\infty]$ be a measurable function.  The pair $(f,g)$  \textit{satisfies a $p$\nobreakdash-Poin\-ca\-r\'e inequality} with constant $C>0$ and dilation factor $\sigma>0$ if for each ball $B$ in $X$,
\begin{equation}\label{Lorentz:PI}
\mint_B |f-f_B| \ d\mu \leq C(\diam B)\left(\mint_{\sigma
B} g^p\, d\mu\right)^\frac{1}{p}.
\end{equation}
The space $(X,d,\mu)$ \textit{supports a $p$-Poincar\'e} inequality if there is a constant $C>0$ and a dilation factor $\sigma >0$ such that for any locally integrable function $f \colon X \to \reals$ and each $p$-weak upper gradient $g$ of $f$, the pair $(f,g)$ satisfies a $p$-Poincar\'e inequality with constant $C$ and dilation factor $\sigma$.

By \cite[Theorem 4.3]{HKST}, if $X$ is doubling and supports a $p$-Poincar\'e inequality, then for any metric space $Y$, any locally integrable mapping $f \colon X \to Y$, any $p$-weak upper gradient $g$ of $f$, and any isometric embedding $\iota$ of $Y$ into a Banach space $V$, the pair $(\iota \circ f,g)$ supports a $p$-Poincar\'e inequality, quantitatively.

\subsection{Measurable differential structures}\label{diff Section}

We briefly outline the theory of differentiation in metric spaces developed by Cheeger \cite{Cheeger99}.  The interested reader could also see the work of Keith \cite{Keith04} and the primer of Kleiner and Mackay \cite{primer}. 

\begin{definition}[measurable differentiable structure]
\label{stepanov:smds} A \emph{measurable differentiable structure} on a metric measure space $(X,d,\mu)$ is a countable collection
of pairs $\left\{\left(X_\alpha,\varphi_\alpha\right)\right\}_{\alpha \in I}$, called \emph{coordinate patches}, that satisfy the following conditions: \begin{itemize}
\item  for each $\alpha \in I$, the set $X_\alpha$ is a measurable subset of $X$ of positive measure,
\item the union $\bigcup_{\alpha} X_\alpha$ is pairwise disjoint and has full measure in $X$,
\item each $\varphi_\alpha$ is an $N(\alpha)$-tuple of Lipschitz functions on $X$, for some $N(\alpha)\in \nats$ that is bounded above independently of $\alpha$,
\item for every Lipschitz function $f\colon X \to \reals$, there exists a collection of measurable functions $\{\partial f/\partial\varphi_n^\alpha\}_{\alpha \in I, n\in \{1,\hdots, N(\alpha)\}}$ such that for each $\alpha \in I$ and for $\mu$-almost-every point $x \in X_\alpha$,
\begin{equation}\label{differential} \lim_{\substack{y\to x\\y\in X_\alpha \bslash \{x\}}}
\frac{\left| f(y)-f(x) - \sum_{n=1}^{N(\alpha)}(\varphi^\alpha_n(y)-\varphi^\alpha_n(x))\frac{\partial{f}}{\partial{\varphi^\alpha_n}}(x)\right|}{d(y,x)}=0,
\end{equation}
and moreover that this condition determines the collection $\{\partial f/\partial\varphi_n^\alpha\}$ unique\-ly up to sets of measure zero.
\end{itemize}
\end{definition}

\begin{definition}[differentiability]  Let $\left\{\left(X_\alpha,\varphi_\alpha\right)\right\}_{\alpha \in I}$ be a measurable differentiable structure on a metric measure space $(X,d,\mu)$ and let $V$ be a Banach space. Given a measurable subset $S$ of $X$, a Bochner measurable mapping $f \colon X \to V$ is \emph{differentiable at $\mu$-almost-every point of $S$} if there exists a collection of measurable functions $\{\partial f/\partial\varphi_n^\alpha\ \colon S \cap X_\alpha \to V\}_{\alpha \in I, n\in \{1,\hdots, N(\alpha)\}}$ such that for each $\alpha \in I$ and for $\mu$\nobreakdash-almost-every  point $x \in S \cap X_\alpha$,
\begin{equation}\label{differentialBanach} \lim_{\substack{y\to x\\y\in X_\alpha \bslash \{x\}}}
\frac{\left\lVert f(y)-f(x) - \sum_{n=1}^{N(\alpha)}(\varphi^\alpha_n(y)-\varphi^\alpha_n(x))\frac{\partial{f}}{\partial{\varphi^\alpha_n}}(x)\right\rVert_V}{d(y,x)}=0,
\end{equation}
and moreover that this condition determines the collection $\{\partial f/\partial\varphi_n^\alpha\}$ uniquely up to sets of measure zero in $S$.
\end{definition}

It is a deep result that many metric measure spaces have a measurable differentiable structure \cite{Cheeger99}. The following statement can be viewed as a Rademacher theorem for metric measure spaces.

\begin{theorem}[Cheeger]\label{StructExist} Let $(X,d,\mu)$ be a doubling metric measure space that supports a $p$-Poincar\'e inequality for some $1\leq p < \infty.$ Then there exists a measurable differentiable structure on $(X,d,\mu)$.
\end{theorem}

We recall that a Banach space $V$ has the \emph{Radon-Nikodym property} if every Lipschitz function $f \colon \reals \to V$ is differentiable almost everywhere with respect to Lebesgue measure. There are several equivalent definition for this property, but  the one given above is the most natural in our context. The following theorem can be viewed as a Rademacher theorem for Banach space-valued mappings on metric spaces \cite{RNP}.

\begin{theorem}[Cheeger-Kleiner]\label{BanachRad} Let $(X,d,\mu)$ be a doubling metric measure space that supports a $p$-Poincar\'e inequality for some $1 \leq p < \infty$, and let $\left\{\left(X_\alpha,\varphi_\alpha\right)\right\}_{\alpha \in I}$ be a measurable differentiable structure on $(X,d,\mu)$. Then every Lipschitz mapping from $X$ to a Banach space with the Radon-Nikodym property is differentiable almost everywhere in $X$.
\end{theorem}

\subsection{Lorentz spaces}
We now introduce the \emph{Lorentz spaces}, which refine the \linebreak Lebesgue spaces.

Given a measure space $(X,\mu)$ and a Banach space
$(V,\norm{\cdot}_V)$, we denote by $\mathcal{M}$ and
$\mathcal{M}_0$ the following classes of functions, respectively:
\begin{align*}
\mathcal{M}&:=\{f\colon X\to V: \text{ $f$
$\mu$\nobreakdash-measurable}\},\\
\mathcal{M}_0&:=\{\norm{f}_V\in \mathcal{M}: \text{$f$ finite
$\mu$\nobreakdash-almost everywhere}\}.
\end{align*}
Given $f\in \mathcal{M}_0$, we define the \emph{distribution
function} $\omega_f\colon [0,\infty)\to [0,\infty]$ and the
\emph{nonincreasing rearrangement} $f^*\colon [0,\infty)\to
[0,\infty]$ by
\begin{align*}
\omega_f(\alpha)&:=\mu(\{x\in X:\, \norm{f(x)}_V>\alpha\}),\\
f^*(t)&:=\inf\{\alpha\geq 0:\, \omega_f(\alpha)\leq t\}.
\end{align*}

Let $1\leq Q\leq
\infty$ and $0<q\leq \infty$. The
$(Q,q)$\nobreakdash-\emph{Lorentz class} consists of those
functions $f\in \mathcal{M}_0(X)$ such that the quantity
\begin{equation*}
\norm{f}_{Q,q}:=
\begin{cases}
(\int_0^\infty (t^{1/Q}f^*(t))^q\, \frac{dt}{t})^{1/q}, &
0<q<\infty, \\
\sup_{0<t<\infty}\{t^{1/Q}f^*(t)\}, & Q<\infty \text{ and }
q=\infty,\\
f^*(0), & q=Q=\infty
\end{cases}
\end{equation*}
is finite. If $1\leq q \leq Q$, then $\norm{\cdot}_{Q,q}$ defines
a semi-norm on the $(Q,q)$\nobreakdash-Lorentz class, and the
corresponding normed space $(L^{Q,q}(X),\norm{\cdot}_{Q,q})$ is a
Banach space. We refer to it as the
$(Q,q)$\nobreakdash-\emph{Lorentz space}.

\section{Proof of Theorem \ref{main}}\label{main proof section}

In this section, we assume the hypotheses of Theorem \ref{main}. Namely, we fix $Q \geq 1$ and $1 \leq q \leq Q$, and let $(X,d,\mu)$ be a complete and Ahlfors $Q$-regular metric space that supports a $q$-Poincar\'e inequality, and let $Y$ be any metric space. We consider a continuous and locally integrable function $f \colon X \to Y$ and the set
$$E =\{x \in X: \lip f(x)=\infty\}.$$

Let us consider the case $Q=1$, which implies that $q=1$. We assume that $\lip f \in L^{1,1}(X)=L^1(X)$, and that $E$ is countable. An easy extension of \cite[Lemma~3.9]{Z} to include metric valued mappings states that  $\lip f$ is a $1$\nobreakdash-weak
upper gradient of $f$ under precisely our assumptions. The general principle stating that the existence of a weak upper gradient in a certain integrability class implies the existence of a (non-weak) upper gradient in the same class (see \cite[Lemma~2.4]{Diamond}) remains valid in this setting, and hence there is an upper gradient of $f$ in every $L^1(X)$ neighborhood of $\lip f$. Since $L^{1}(X)=L^{1,1}(X)$, this suffices.

We now assume that $Q>1$. In the case that $q>1$, the open-endedness result of Keith and Zhong \cite[Theorem~1.0.1]{KeithZhong} states that $X$ actually supports a $q'$\nobreakdash-Poincar\'e inequality for some $q'<q$, quantitatively. Since a set that has Hausdorff dimension at most $(Q-q)$ has zero $(Q-q')$-dimensional Hausdorff measure, we assume without loss of generality that $1 \leq q < Q$, that $\lip f \in L^{Q,1}(X)$, and that the set $E$ has $\sigma$-finite $(Q-q)$-dimensional Hausdorff measure. We wish to show that $f$ has an upper gradient in the space $L^{Q,1}(X)$. Again,  \cite[Lemma~3.9]{Z} implies that $\lip f$ is a $q$-weak upper gradient of $f$, and hence the pair $(f,\lip f)$ satisfies the $q$-Poincar\'e inequality.

It is a well-known principle that if a function-gradient pair $(f,g)$ satisfies the $q$-Poincar\'e inequality, then the perturbed maximal function
$$ M_q(g)(x) := \left(\sup_{r > 0} \mint_{B(x,r)} g^q \ d\mu\right)^{1/q}$$
provides a pointwise bound almost everywhere on the oscillation of $f$, i.e., $M_q(g)$ is a \emph{Haj\l asz upper gradient} of $f$. See \cite[Theorem~3.2]{SobMetPoin} and \cite[Proposition~4.6]{HKST}. Accordingly, there is a constant $C \geq 1$, depending only on the data, and a set $N$ of $\mu$-measure zero such that for each pair of points $x,y \in X\bslash N$,
\begin{equation}\label{hug} d_Y(f(x),f(y)) \leq Cd_X(x,y)(M_q(\lip f)(x)+M_q(\lip f)(y)).\end{equation}

The Hardy-Littlewood maximal function theorem \cite[Theorem~2.2]{LAMS} and the Marcinkiewicz Interpolation Theorem \cite[Theorem~IV.4.13]{BS} can be shown to imply the boundedness of the operator $M_q \colon L^{Q,1}(X)\to
L^{Q,1}(X)$.
This was essentially stated in \cite{romanov_absolute_2008} and proven in detail in \cite[Section~4]{Paper3}.

The remainder of the proof is based on \cite[Lemmas 4.6 and 4.7]{Shanmugalingam00}. Define $g \colon X \to [0,\infty]$ by
$$g(x) =\begin{cases} CM_q(\lip f)(x) & x \notin N, \\
		      \infty & x \in N.\\
	\end{cases}$$
Then
$$||g||_{Q,1} \lesssim ||M_q(\lip f)||_{Q,1} \lesssim ||\lip f||_{Q,1},$$
and the inequality \eqref{hug} implies that for every pair of points $x$ and $y$ in $X$,
\begin{equation}\label{all hug}d_Y(f(x),f(y)) \leq d_X(x,y)(g(x)+g(y)).\end{equation}

We now complete the proof by showing that $4g$ is an upper gradient of $f$. It suffices to show that given an arc-length parameterized path $\gamma \colon [0,L] \to X$, it holds that
\begin{equation*}
d_Y(f(\gamma(0)),f(\gamma(L))) \leq \int_{\gamma} g \ ds.
\end{equation*}
We may assume that the integral of $g$ over $\gamma$ is
finite.

Fix $n \in \nats$, and for each $k\in \{0,1,\ldots, n-1\}$ define $\gamma_k = \gamma|_{I_k}$, where
$$I_k:=\left[ \frac{kL}{n},\frac{(k+1)L}{n}\right].$$
For each such $k$, there exits $x_k\in \gamma(I_k)$ such that
\begin{equation*}
g(x_k)\leq \frac{1}{\length{\gamma_k}}\int_{\gamma_k} g\, ds.
\end{equation*}
It follows that
$$d(x_k,x_{k+1})\leq \length(\gamma_k)+\length(\gamma_{k+1})=\frac{2L}{n}.$$
Using the inequality \eqref{all hug}, we see that
\begin{align*}
d_Y(f(x_0),f(x_{n-1}))&\leq \sum_{k=0}^{n-2} d_Y(f(x_k),f(x_{k+1}))
\leq \sum_{k=0}^{n-2} d(x_k,x_{k+1})(g(x_k)+g(x_{k+1})) \\
&\leq \sum_{k=0}^{n-2} \frac{2L}{n}\left(\frac{1}{\length(\gamma_k)} \int_{\gamma_k} g\, ds + \frac{1}{\length(\gamma_{k+1})}\int_{\gamma_{k+1}} g\, ds \right) \\
&\leq 4\sum_{k=0}^{n-1} \int_{\gamma_k} g\, ds = 4\int_{\gamma} g\, ds.
\end{align*}
Since $f$ and $\gamma$ are continuous, letting $n$ tend to infinity yields
\begin{equation*}
d_Y(f(\gamma(0)),f(\gamma(L))) \leq 4\int_{\gamma} g\, ds
\end{equation*}
as desired.

\section{The proof of Theorem \ref{Stein}} \label{Stein section}
In order to prove Theorem \ref{Stein}, we first show that the mapping $f$ in question satisfies $\Lip f(x)<\infty$ for almost every $x \in X$, and then apply a Stepanov-type theorem in the spirit of Cheeger and Kleiner.  The first step follows from \cite{SpaceFilling}. The second step is based on the Cheeger-Kleiner generalization of Rademacher's theorem (Theorem \ref{BanachRad} above), and the usual method of proving Stepanov's theorem from Rademacher's theorem (see, for example, \cite[3.1.9]{Federer}). Of novelty here is the use of Lang and Schlichenmaier's Lipschitz extension results.

For a mapping $f \colon X \to Y$ between metric spaces, we denote 
$$S_f = \{x \in X: \Lip f(x)<\infty\}.$$
The following lemma is standard.

\begin{lemma}\label{partition} Let $(X,d,\mu)$ be a separable metric measure space, $(Y,d_Y)$ any metric space, and $f \colon X \to Y$ a mapping. Then there is a sequence $\{C_k\}_{k \in \nats}$ of closed subsets of $X$ such that for each $k \in \nats$, the restriction $f|_{C_k}$ is Lipschitz and 
$$S_f=\bigcup_{k \in \nats} C_k$$
\end{lemma} 

\begin{proof} For each $k \in \nats$, we define
$$A_k = \left\{x \in X: d_Y(f(x),f(y)) \leq kd_X(x,y) \ \text{for all}\ y \in B_X(x,1/k) \right\}.$$
Note that
$$S_f = \bigcup_{k=1}^\infty A_k.$$

We claim that each set $A_k$ is closed. Consider a sequence $x_n \in A_k$ and a point $x \in X$ such that $d(x,x_n)$ tends to $0$. Given a point $y \in B_X(x,1/k)$, there is a number $N \in \nats$ such that if $n \geq N$, then both $y$ and $x$ are in the ball $B(x_n,1/k)$. Hence, for all $n \geq N$,
$$d_Y(f(x),f(y)) \leq d_Y(f(x),f(x_n)) + d_Y(f(x_n),f(y)) \leq k(d_X(x,x_n) + d_X(x_n,y)).$$
Letting $n$ tend to $\infty$ shows that $x \in A_k$ and proves the claim.

As $X$ is separable, we may cover the set $A_k$ by a countable collection of closed balls $\{B_{k,j}\}_{j\in J_k}$ of diameter strictly less than $1/{2k}$. We now claim that for any $j \in J_k$, the map $f$ is $k$-Lipschitz in the closed set $A_k \cap B_{k,j}$. Indeed, given $x $ and $y$ in $A_k \cap B_{k,j}$, we see that $y \in B_X(x,1/k)$, and hence the definition of $A_k$ provides the desired inequality. Re-indexing the countable collection 
$$\{A_k \cap B_{k,j} : k \in \nats, j \in J_k\}$$
now completes the proof. 
\end{proof} 


\begin{theorem}\label{Stepanov} Let $(X,d,\mu)$ be a complete and doubling metric measure space satisfying a $p$-Poincar\'e inequality for some $p \in [1,\infty)$, and let $V$ be a Banach space with the Radon-Nikodym property. Then each measurable function $f\colon X \to V$ is differentiable at almost every point of the set $S_f=\{x \in X: \Lip f(x)<\infty\}$ with respect to any Cheeger structure on $X$.
\end{theorem}

\begin{proof} Let $\{C_k\}_{k \in \nats}$ be the sequence of closed sets provided by Lemma \ref{partition}. By the Lipschitz extension theorem of Lang and Schlichenmaier \cite[Theorem~1.5]{Lang}, there is a Lipschitz mapping $F_{k} \colon X \to V$ such that the restriction of $F_{k}$ to $C_k$ agrees with $f$.

Let $\{(X_\alpha,\varphi^\alpha)\}_{\alpha \in I}$ be a measurable differential structure on $X$. By the Cheeger-Kleiner generalization of Rademacher's theorem \cite[Theorem~1.5]{RNP}, the mapping $F_{k}$ is differentiable at almost every point of $X$ with respect to this structure. Let
\begin{equation}\label{partial derivs F} \{\partial{F_{k}}/\partial{\varphi^\alpha_n}\colon X_\alpha \to V\}_{\alpha \in I, n=1,\hdots, N(\alpha)}\end{equation}
be the collection of partial derivatives of $F_{k}$ over all coordinate patches.

In order to show the differentiability of $f$ almost everywhere in the set $S_f$, we will show that the countable collection
\begin{equation} \label{partial derivs f}\{\left(\partial{F_{k}}/\partial{\varphi^\alpha_n}\right)|_{S_f} \colon S_f \cap X_\alpha \to V\}_{\alpha \in I, n=1,\hdots, N(\alpha)}\end{equation}
consists of Borel measurable functions satisfying, for each $\alpha \in I$, almost every point $x_0$ in \mbox{$S_f \cap X_\alpha$}, any sequence $\{y_m\}_{m \in \nats}$ of points in $X_\alpha\bslash \{x_0\}$ that converges to $x_0,$
\begin{equation}\label{diff lim}\limsup_{m \to \infty} \frac{\left\lVert f(y_m)-f(x_0) - \sum_{n=1}^{N(\alpha)}(\varphi^\alpha_n(y_m)-\varphi^\alpha_n(x_0))\frac{\partial{F_{k,j}}}{\partial{\varphi^\alpha_n}}(x_0)\right\rVert_V}{d(y_m,x_0)}=0, \end{equation}
and moreover that up to a set of measure zero in $S_f$, this collection is uniquely specified.

The Borel measurability of the functions in \eqref{partial derivs f} follows from the corresponding statement for the functions in \eqref{partial derivs F} and the decomposition of $S_f$ into a countable union of closed sets. By this decomposition, the Lebesgue differentiation theorem \cite[Theorem~1.8]{LAMS}, and the almost everywhere differentiability of the functions $\{F_{k}\}$, we need only show \eqref{diff lim} for a point $x_0$ that is both a density point of  $C_k \cap X_\alpha$ and a point of differentiability of $F_{k}$, for some $k \in \nats$ and $\alpha \in I$. 

Let $\{y_m\}_{m \in \nats} \subeq X_\alpha$ be a sequence converging to $x_0$, and fix $0<\ep<1$. As noted on \cite[Page 409]{BRZ07}, since $x_0$ is a density point for $C_k \cap X_\alpha$, there is a radius $0<r<1/(2k)$ such that if $y_m \in B(x_0,r)$, then there is a point $x_m \in C_k \cap X_\alpha$ satisfying
\begin{equation}\label{density}d(y_m,x_m) \leq \ep d(y_m,x_0) < 2r < 1/k.\end{equation}
We may assume without loss of generality that $y_m \in B(x_0,r)$ for every $m \in \nats$.  For each $m \in \nats$, define
$$L_1(m) = \frac{\left\lVert f(y_m)-f(x_m) - \sum_{n=1}^{N(\alpha)}(\varphi^\alpha_n(y_m)-\varphi^\alpha_n(x_m))\frac{\partial{F_{k}}}{\partial{\varphi^\alpha_n}}(x_0)\right\rVert_V}{d(y_m,x_0)},$$
$$L_2(m) = \frac{\left\lVert f(x_m)-f(x_0) - \sum_{n=1}^{N(\alpha)}(\varphi^\alpha_n(x_m)-\varphi^\alpha_n(x_0))\frac{\partial{F_{k}}}{\partial{\varphi^\alpha_n}}(x_0)\right\rVert_V}{d(y_m,x_0)}.$$

We claim that
\begin{equation}\label{limsup}\limsup_{m \to \infty} \left(L_1(m)+L_2(m) \right) = 0,\end{equation}
which implies \eqref{diff lim} by the triangle inequality. Fix $m \in \nats$. Recall that $f|_{C_k}$ is $L$-Lipschitz, for some $L \geq 1$; without loss of generality, we may assume that each chart mapping $\varphi_n^\alpha$ is also $L$-Lipschitz. By the triangle inequality, \eqref{density}, and the fact that $x_m \in C_k$,
\begin{align*}L_1(m) \leq & \ep\left(\frac{||f(y_m)-f(x_m)||_V}{d(y_m,x_m)} +\frac{\left(\sum_{n=1}^{N(\alpha)}|\varphi^\alpha_n(y_m)-\varphi^\alpha_n(x_m)|\right) \left\lVert\frac{\partial{F_{k}}}{\partial{\varphi^\alpha_n}}(x_0)\right\rVert_V}{d(y_m,x_m)} \right) \\ &\leq \ep \left( L + \sum_{n=1}^{N(\alpha)}L \left\lVert\frac{\partial{F_{k}}}{\partial{\varphi^\alpha_n}}(x_0)\right\rVert_V \right).
\end{align*}
Note that the final quantity above is independent of $m$, and tends to zero as $\ep$ tends to zero.

Moreover, the triangle inequality, the inequalities \eqref{density}, and the fact that $x_m$ and $x_0$ are points of $C_k \cap B_j$ imply that
\begin{align*} L_2(m) &\leq  \frac{\left\lVert f(x_m)-f(x_0) - \sum_{n=1}^{N(\alpha)}(\varphi^\alpha_n(x_m)-\varphi^\alpha_n(x_0))\frac{\partial{F_{k}}}{\partial{\varphi^\alpha_n}}(x_0)\right\rVert_V}{d(x_m,x_0)} \\ & \hspace{2cm} \cdot \frac{d(x_m,y_m)+d(y_m,x_0)}{d(y_m,x_0)} \\
& \leq \frac{\left\lVert F_{k}(x_m)-F_{k}(x_0) - \sum_{n=1}^{N(\alpha)}(\varphi^\alpha_n(x_m)-\varphi^\alpha_n(x_0))\frac{\partial{F_{k}}}{\partial{\varphi^\alpha_n}}(x_0)\right\rVert_V}{d(x_m,x_0)}\cdot(1+\ep).
\end{align*}
Since $F_{k,j}$ is differentiable at $x_0$, the above quantity tends to zero as $m$ tends to infinity. This completes the proof of \eqref{limsup} and hence of \eqref{diff lim}.

We now show that the collection \eqref{partial derivs f} is unique up to a set of measure zero in $S_f$. By the uniqueness of the collection \eqref{partial derivs F}, it suffices to show that if $x_0$ is a point of density of some $C_k \cap X_\alpha$, and $\{g_n\}_{n=1}^{N(\alpha)}$ is a collection of Borel measurable functions satisfying, for any sequence $\{y_m\} \subeq X_\alpha$ converging to $x_0$,
\begin{equation}\label{diff lim 1}\limsup_{m \to \infty} \frac{\left\lVert f(y_m)-f(x_0) - \sum_{n=1}^{N(\alpha)}(\varphi^\alpha_n(y_m)-\varphi^\alpha_n(x_0))g_n(x_0)\right\rVert_V}{d(y_m,x_0)}=0, \end{equation}
then it also holds that
\begin{equation}\label{diff lim 2}\limsup_{m \to \infty} \frac{\left\lVert F_k(y_m)-F_k(x_0) - \sum_{n=1}^{N(\alpha)}(\varphi^\alpha_n(y_m)-\varphi^\alpha_n(x_0))g_n(x_0)\right\rVert_V}{d(y_m,x_0)}=0. \end{equation}
The proof of this is analogous to the proof of \eqref{diff lim} and is left to the reader.
\end{proof}

\begin{proof}[Proof of Theorem \ref{Stein}] We fix $Q \geq 1$, and let $(X,d,\mu)$ be a complete and Ahlfors $Q$-regular metric space that supports a $Q$-Poincar\'e inequality. Moreover, let $V$ be a Banach space that has the Radon-Nikodym property, i.e., every Lipschitz function from $\reals$ to $V$ is differentiable almost everywhere with respect to the Lebesgue measure on $\reals$. We consider a continuous $f \colon X \to V$ that is assumed to have an upper gradient in the Lorentz space $L^{Q,1}$.

It follows from \cite[Corollary~6.7]{SpaceFilling} that $f$ satisfies the \emph{$Q$-Rado-Reichelderfer condition}, meaning that there is a number $\sigma >0$ and a non-negative function $\Theta \in L^1(X)$ such that for any ball $B$ in $X$,
$$\diam(f(B))^Q \leq \int_{\sigma B} \Theta \ d\mu.$$
An easy computation now shows that $\Lip f(x)<\infty$ for almost every $x \in X$ \cite[Proposition~6.4]{SpaceFilling}. Theorem \ref{Stepanov} completes the proof.
 \end{proof}

\section{Non-embedding results}\label{weak tan section}
We now apply the results of the previous sections to prove a non-embedding result, Theorem \ref{RNP non}. The material in the preparatory first two subsections might be familiar to experts. 

\subsection{Existence and basic properties of weak tangent mappings}\label{weak tan section1}

We begin with the existence and properties of a weak tangent mapping of a bi-Lipschitz embedding.  This construction is standard and can be accomplished in several different but equivalent ways; see, for example, \cite{Cheeger99} or \cite{Burago}. We include it for convenience and because we will later refer to specific parts of the construction. We follow the method of David and Semmes \cite{david_fractured_1997}. This approach has the advantage of being concrete although somewhat cumbersome in notation; it uses the Assouad embedding theorem to transfer the abstract notion of pointed Gromov-Hausdorff convergence to the Euclidean setting. In Euclidean space, the relevant notions of convergence are as follows.

A sequence $\{F_j\}_{j\in \nats}$ of non-empty closed subsets of $\reals^N$ \emph{converges} to a non-empty closed subset $F$ of $\reals^N$ if for all $r>0$,
\begin{align*} & \lim_{j \to \infty} \sup\{\dist(x,F): x \in F_j \cap B(0,r)\} = 0,\ \rm{and} \\
 & \lim_{j \to \infty} \sup\{\dist(x,F_j): x \in F \cap B(0,r)\}=0,
 \end{align*}
where we use the convention $\sup \emptyset=0$. 

A sequence of functions $\{f_j \colon F_j \to \reals^N\}_{j \in \nats}$ \emph{converges} to a function $f \colon F \to \reals^N$ if the sequence of sets $\{F_j\}_{j\in \nats}$ converges to the set $F$ and for every convergent sequence $\{x_j \in F_j\}_{j \in \nats}$,
 $$f\left(\lim_{j\to \infty} x_j\right) =\lim_{j \to \infty} f_j(x_j).$$

We recall that a \emph{pointed metric space} is a triple consisting of a set, a metric on the set, and a point in the set (called the \emph{base point}). We abuse notation by denoting, for $N \in \nats$, the pointed metric space $(\reals^N,||\cdot||_{\reals^N},0)$ by $\reals^N$. A \emph{pointed mapping package} is a triple $(A,B,f)$ where $A$ and $B$ are pointed metric spaces, and $f$ is a function between the sets underlying $A$ and $B$ that preserves the base point.

We now assume that $(X,d_X,x_0)$ is a complete and doubling pointed metric space, and that $C \subeq X$ is a closed set containing $x_0$. We denote by $(V,||\cdot||_V)$ a Banach space. We consider an embedding $\iota \colon X \into V$, and we assume that there is a number $A \geq 1$ such that the restriction $\iota|_{C}$ is an $A$-bi-Lipschitz embedding of $C$ into $V$.

Take any sequence of scales $\{r_j\}_{j \in \nats}$ tending to zero.  For each $j \in \nats$, define pointed metric spaces
\begin{align*} C_j &= (C,d_X/r_j,x_0), \\   \iota(C)_j &= (\iota(C),||\cdot||_V/r_j,\iota(x_0)).
\end{align*}
Note that these spaces are again complete and doubling with a uniform constant.  Gromov's compactness theorem, Lemma~8.13 in \cite{david_fractured_1997}, implies that for some subsequence of the scales $\{r_j\}_{j\in \nats}$ (which we do not relabel) the corresponding sequences of pointed metric spaces converge, and we denote the limit spaces by
\begin{align*}
C_{x_0} = (C_\infty,d_\infty,x_\infty)\ \text{and}\ 
 \iota(C)_{\iota(x_0)} = (\iota(C)_\infty, \rho_\infty,y_\infty),
\end{align*}
respectively. The space $C_{x_0}$ is called a \emph{weak tangent} to $C$ at $x_0$, and $\iota(C)_{\iota(x_0)}$ is named analogously. 

After passing to another subsequence of $\{r_j\}_{j\in \nats}$ (again, without relabeling), we may assume that pointed mapping packages $\{(C_j,\iota(C)_j,\iota)\}_{j \in \nats}$ converge to a pointed mapping package $(C_{x_0},\iota(C)_{\iota(x_0)},\iota_\infty)$; see Lemma~8.22 in \cite{david_fractured_1997}.

The meaning of the above convergence is the following:
\begin{itemize}
\item there are dimensions $N, M \in \nats$, exponents $\alpha, \beta \in (0,1]$, constants $K, L \geq 1$, sequences of pointed mapping packages
\begin{align*} \{ & (C_j,\reals^N,f_j)\}_{j \in \nats}, \ \ \ \{(\iota(C)_j,\reals^M, g_j)\}_{j \in \nats}, \\
\intertext{and ``limit" pointed mapping packages} &
(C_{x_0},\reals^N , f), \ \ \ \ \ \ \ \ \ (\iota(C)_{\iota(x_0)},\reals^M,g),
\end{align*}
where each of $\{f_j\}_{j \in \nats}$ and $f$ is $(\alpha, K)$-bi-H\"older, and each of $\{g_j\}_{j \in \nats}$ and $g$ \ is $(\beta, L)$-bi-H\"older.
\item the sequences of subsets $\{f_j(C_j)\}_{j \in \nats}$ and $\{g_j(\iota(C_j))\}_{j \in \nats}$ converge to $f(C_{x_0})$ and $g(\iota(C)_{\iota(x_0)})$, respectively,
\item the maps 	
\begin{align*} \frac{d_X(f_j\inv(\cdot),f_j\inv(\cdot))}{r_j} &\colon f_j(C_j) \times f_j(C_j) \to \reals, \\
					 \frac{||g_j\inv(\cdot)-g_j\inv(\cdot)||_{V}}{r_j} &\colon g_j(\iota(C)_j) \times g_j(\iota(C)_j) \to\reals, \\
					 g_j \circ \iota \circ f_j\inv &\colon f_j(C) \to g_j(\iota(C)),
					 \end{align*}
converge to
\begin{align*} d_\infty(f\inv(\cdot),f\inv(\cdot)) &\colon f(C_{x_0}) \times f(C_{x_0}) \to \reals, \\
					 \rho_\infty(g\inv(\cdot),g\inv(\cdot)) & \colon  g(\iota(C)_{\iota(x_0)}) \times g(\iota(C)_{\iota(x_0)}) \to \reals, \\
					 g \circ \iota_\infty \circ f\inv &\colon f(C_{x_0}) \to g(\iota(C)_{\iota(x_0)}),\end{align*}
respectively.
\end{itemize}
It can easily be seen that the limit map $\iota_\infty \colon C_{x_0} \to \iota(C)_{\iota(x_0)}$ is also an $A$-bi-Lipschitz embedding. Moreover, the doubling condition on $X$, and hence $C$, persists in the limit: $C_{x_0}$ and thus $\iota(C)_{\iota(x_0)}$ are also doubling.  

\subsection{Properties of weak tangent mappings at points of differentiability}\label{weak tan section2}
We now additionally assume the existence of a measure $\mu$ on $(X,d_X)$ so that the resulting metric measure space $(X,d_X,\mu)$ has a measurable differentiable structure $\{(X_\alpha,\varphi_\alpha)\}_{\alpha \in I}$, and that the embedding $\iota \colon X \into V$ is differentiable almost everywhere in $C$, with differentials $\{\frac{\partial \iota}{\partial \varphi_n^\alpha} \colon  (C \cap X_{\alpha}) \to V\}_{\alpha \in I, n \in \{1,\hdots, N(\alpha)\}}.$ Recall that this means that for each $\alpha \in I$ and $\mu$-almost every point $x_0 \in C \cap X_\alpha$,
\begin{equation}\label{diff reminder} \lim_{\substack{x\to x_0\\x\in X_\alpha \bslash \{x_0\}}}
\frac{\left| \iota(x)-\iota(x_0) - \sum_{n=1}^{N(\alpha)}(\varphi^\alpha_n(x)-\varphi^\alpha_n(x_0))\frac{\partial{\iota}}{\partial{\varphi^\alpha_n}}(x_0)\right|}{d(x,x_0)}=0,
\end{equation}
and moreover that this condition determines the collection $\{\frac{\partial \iota}{\partial \varphi_n^\alpha}\}_{n=1,\hdots,N(\alpha)}$ unique\-ly up to sets of measure zero in $C \cap X_{\alpha}$.

 The following proposition, which follows immediately from the definitions, says that given a point $x_0 \in C \cap X_\alpha$ at which \eqref{diff reminder} holds,  the image $\iota(C)$ is \emph{contained up to first order in the span of the differentials of $\iota$ at $x_0$}, a concept introduced by Cheeger and Kleiner \cite{RNP}.  We denote by
 $F_{x_0}$ the span in $V$ of the vectors $\{\frac{\partial \iota}{\partial \varphi_n^\alpha}(x_0)\}_{n=1,\hdots,N(\alpha)}$.
\begin{proposition}\label{finite dim} Let $x_0 \in C \cap X_\alpha$ be a point at which \eqref{diff reminder} holds. Then
$$\limsup_{r \to 0} \sup_{x \in B_X(x_0,r) \cap C} \frac{\dist_V(\iota(x)-\iota(x_0), F_{x_0})}{r}=0.$$
\end{proposition}

We now use this information to bound the dimension of the tangent spaces to $C$ by the dimension of the differentiable structure on $X$. When $C=X$ this type of result can be found in \cite[Section~14]{Cheeger99} and is also mentioned in \cite{NotRNP}, although we have not been able to find an explicit proof of the proposition below in the literature. 

\begin{proposition}\label{dim bound}  Let $x_0 \in C \cap X_\alpha$ be a point at which \eqref{diff reminder} holds. Then the weak tangent space $\iota(C)_{\iota(x_0) }$ isometrically embeds into $F_{x_0}$.
\end{proposition}

\begin{proof} 
Let $y$ be a point in the weak tangent space $\iota(C)_{\iota(x_0)}$. By definition, we may find a sequence of points $x_j \in C$ such that $\lim_{j \to \infty} g_j(\iota(x_j)) = g(y) \in \reals^M$.  As the mapping in a pointed mapping package preserves base points, $g_j(\iota(x_0))=0 \in \reals^M$ for all $j \in \nats$. This implies that
$$||g(y)||_{\reals^M} = \lim_{j \to \infty} ||g_j(\iota(x_j)) - g_j(\iota(x_0))||_{\reals^M} \geq \lim_{j \to \infty} \frac{1}{L}\left(\frac{||\iota(x_j) -\iota(x_0)||_{V}}{r_j}\right)^\beta.$$
This shows that the sequence $\left\{\frac{\iota(x_j) -\iota(x_0)}{r_j}\right\}_{j \in \nats}$ is bounded in $(V,||\cdot||_V).$  Moreover, there is a quantity $c = c(A, L, \beta, y)$, such that for all $j \in \nats$,
$$d_{X}(x_j,x_0) \leq cr_j.$$ 

Let $k\in \nats$. Proposition~\ref{finite dim} now implies the existence of a positive integer $J(k)$ such that for all $j \geq J(k)$, there is a point $y_j \in F_{x_0}$ satisfying 
\begin{equation}\label{consequence}
\norm{\iota(x_j)-\iota(x_0)-y_j}_V\leq \frac{1}{k}r_j.
\end{equation}
Without loss of generality, we may assume that $J(k)<J(k+1)$ for all $k\in \nats$. By passing to a subsequence, we may further assume that $r_j=r_{J(j)}$. The triangle inequality now implies that the sequence $\{y_j/r_j\}_{j \in \nats}$ is bounded. As each element of the sequence is in the finite dimensional subspace $F_{x_0}$, after passing to a subsequence we may assume that  $\{y_j/r_j\}_{j \in \nats}$ converges to a point $\ovl{y} \in F_{x_0}$. It now follows from \eqref{consequence} that 
$$\lim_{j \to \infty} \frac{\iota(x_j)-\iota(x_0)}{r_j} = \ovl{y} \in F_{x_0}$$
as well.



We wish to define a map $G \colon \iota(C)_{\iota(x_0)} \to F_{x_0}$ by $G(y)=\ovl{y}$. Unfortunately, given points $y$ and $w$ in $\iota(C)_{\iota(x)}$, the subsequences of the scales $\{r_j\}_{j\in \nats}$ used to define $\ovl{y}$ and $\ovl{w}$ might not coincide.  We overcome this using the separability of the doubling space $\iota(C)_{\iota(x_0)}$. Let $\{z_i\}_{i \in \nats}$ be a countable dense set in $\iota(C)_{\iota(x_0)}$.  As argued above, we may find an infinite subset $S_1 \subeq \nats$, a sequence $\{x^1_j\}_{j \in S_1} \subeq C$, and a point $\ovl{z}_1 \in F_{x_0}$ such that 
\begin{equation}\label{G def} g(z_1)=\lim_{\substack{j \to \infty \\ j \in S_1}}g_j(\iota(x^1_j))\ \text{and}\  \lim_{\substack{j \to \infty \\ j \in S_1}} \frac{\iota(x^1_j)-\iota(x_1)}{r_j} = \ovl{z}_1\in F_{x_0}\end{equation}
We now define $G(z_1)=\ovl{z}_1.$ 
We continue inductively, defining $G(z_{i+1})$ similarly but requiring that $S_{i+1} \subeq S_{i}$.  

We now check that $G$ is an isometric embedding of the set $\{z_i\}_{i \in \nats}$ into $F_{x_0}$. Choose $i < k$ in $\nats$, and let $\{x^i_j\}_{j \in S_i}$ and $\{x^k_j\}_{j \in S_k}$ be the sequences in $C$ used to define $G(z_i)$ and $G(z_k)$ respectively.  Since $S_k \subeq S_i$, we have 

\begin{align*} g(z_i)=\lim_{\substack{j \to \infty \\ j \in S_k}}g_j(\iota(x^i_j)) & \ \text{and}\ G(z_i)=\lim_{\substack{j \to \infty \\ j \in S_k}}\frac{\iota(x^i_j)-\iota(x_0)}{r_j}, \\
					g(z_k)=\lim_{\substack{j \to \infty \\ j \in S_k}}g_j(\iota(x^k_j)) & \ \text{and}\ G(z_k)=\lim_{\substack{j \to \infty \\ j \in S_k}}\frac{\iota(x^k_j)-\iota(x_0)}{r_j}. \end{align*}
Since the functions
$$\left\{\frac{||g_j\inv(\cdot)-g_j\inv(\cdot)||_{V}}{r_j} \colon g_j(\iota(C)_j) \times g_j(\iota(C)_j) \to\reals\right\}_{j \in \nats}$$
converge to the function
$$\rho_\infty(g\inv(\cdot),g\inv(\cdot))  \colon  g(\iota(C)_{\iota(x_0)}) \times g(\iota(C)_{\iota(x_0)}) \to \reals,$$
we see that
$$\rho_\infty(z_i,z_k) = \lim_{\substack{j \to \infty \\ j \in S_k}} \frac{||\iota(x^i_j)-\iota(x^k_j)||_{V}}{r_j} = ||G(z_i)-G(z_k)||_V,$$
showing that $G$ is an isometric embedding. Since $\{z_i\}_{i \in \nats}$ is dense in $\iota(C)_{\iota(x_0)}$, the map $G$ extends uniquely to an isometric embedding of $\iota(C)_{\iota(x_0)}$ into the complete space $F_{x_0}$. \end{proof}

Since the tangent mapping $\iota_\infty \colon C_{x_0} \to \iota(C)_{\iota(x_0)}$ induced by $\iota \colon C \into V$ is bi-Lipschitz, we may record the following corollary.
\begin{corollary}\label{embedding to fd} Let $x_0 \in C \cap X_\alpha$ be a point at which \eqref{diff reminder} holds. Then the weak tangent space $C_{x_0}$ admits a bi-Lipschitz embedding into $F_{x_0}$.
\end{corollary}

We note that in fact the crucial property of $\iota$ used in proving Proposition \ref{dim bound} and Corollary \ref{embedding to fd} was not the differentiability \eqref{diff reminder} at the point $x_0$, but rather the fact that the image $\iota(C)$ is contained upto first order at $x_0$ in a finite dimensional space. 

\subsection{The proof of Theorem \ref{RNP non}}
We consider a complete and Ahlfors $Q$-regular metric space $(X,d,\mu)$ that supports a $Q$-Poincar\'e inequality, $Q \geq 1$, and a Banach space $V$ with the Radon-Nikodym property. We equip $(X,d,\mu)$ with a measurable differentiable structure 
$\{(X_\alpha,\varphi_\alpha)\})_{\alpha \in I}$. 

Let $\iota \colon X \to V$ be an embedding with an upper gradient in the Lorentz space $\eL^{Q,1}(X)$. As discussed in the proof of Theorem \ref{main},  \cite[Section~6]{SpaceFilling} implies that
 \begin{equation}\label{iota}\Hdim^Q_{X}(X\bslash \mathcal{S}_{\iota})=0 \ \text{and}\  \Hdim^Q_{V}(\iota(X)\bslash \iota(\mathcal{S}_{\iota}))=0.\end{equation}
In addition, the hypotheses of Theorem \ref{RNP non} state that 
\begin{equation}\label{iotainv}\Hdim^Q_V(\iota(X)\bslash \mathcal{S}_{\iota\inv})=0 \ \text{and} \ \Hdim^Q_{X}(X \bslash \iota\inv(\mathcal{S}_{\iota\inv}))=0.\end{equation}

By Lemma \ref{partition}, we may find a countable collection of closed sets $\{C_k\}_{k \in \nats}$ in $X$ such that $\iota|_{C_k}$ is Lipschitz and 
$$\mathcal{S}_\iota = \bigcup_{k \in \nats} C_k.$$
Since $X$ is doubling and hence separable, the space $\iota(X)$ is also separable. Again using Lemma~\ref{partition}, we find a countable collection of closed sets $\{D_l\}_{l \in \nats}$ in $Y$ such that $\iota\inv|_{D_l}$ is Lipschitz and 
$$\mathcal{S}_{\iota\inv} = \bigcup_{l \in \nats} D_l.$$

Fix $k \in \nats$. We note that $\iota(C_k)$ is closed in the separable space $\iota(X)$, as $\iota$ is an embedding. For each $l \in \nats$, set
$$D_{k,l} = \iota(C_{k}) \cap D_l \ \text{and}\ C_{k,l}=\iota\inv(D_{k,l}).$$
Then, for each pair $k,l \in \nats$, the set $C_{k,l}$ is closed and $\iota|_{C_{k,l}}$ is bi-Lipschitz. Moreover, \eqref{iota} and \eqref{iotainv} imply that 
$$\Hdim^Q_{X}\left( X \bslash \bigcup_{k,l \in \nats} C_{k,l} \right) =0.$$

Let $(X_\alpha,\varphi_\alpha)$ be a coordinate patch of minimal dimension $N(\alpha)=N$. Then we may find $k, l \in \nats$ so that 
$$\Hdim^Q_{X}(C_{k,l} \cap X_\alpha)>0.$$
Theorem \ref{Stepanov} implies that $\iota$ is differentiable at almost every point $\mathcal{S}_{\iota}$, and so we may find a point $x_0 \in C_{k,l} \cap X_{\alpha}$ that is both a density point of $C_{k,l}$ and a point of differentiability of $\iota.$ 

From the discussion in Section \ref{weak tan section1}, we may find some sequence of scales $\{r_j\}$ tending to zero such that 
the pointed metric spaces 
$$X_j = (X,d_X/r_j,x_0)$$
converge in the sense described in Section \ref{weak tan section1} to a weak tangent which we denote by $X_{x_0}$. Since $X$ is Ahlfors $Q$-regular, it follows that $X_{x_0}$, when equipped with the $Q$-dimensional Hausdorff measure, is Ahlfors $Q$-regular as well \cite[Lemma~9.7]{david_fractured_1997}. 

For ease of notation, we denote $C_{k,l}$ simply by $C$. Define pointed metric spaces 
\begin{align*} C_j &= (C,d_X/r_j,x_0), \\   \iota(C)_j &= (\iota(C),||\cdot||_V/r_j,\iota(x_0)).
\end{align*}
As discussed in Section \ref{weak tan section1}, after passing to a subsequence of the scales $\{r_j\}$ (which we do not re-label), 
the pointed mapping packages $\{(C_j,\iota(C)_j,\iota)\}_{j \in \nats}$ converge to a pointed mapping package $(C_{x_0},\iota(C)_{\iota(x_0)},\iota_\infty)$. Moreover, the mapping $\iota_\infty$ is bi-Lipschitz. By \cite[Lemmas~9.12 and 9.13]{david_fractured_1997}, the fact that $x_0$ is a density point of $C$ implies that $C_{x_0}$ and $X_{x_0}$ are isometric. Hence, $C_{x_0}$ is also Ahlfors $Q$-regular when equipped with the $Q$-dimensional Hausdorff measure. Corollary \ref{embedding to fd} now implies that $C_{x_0}$ isometrically embeds into the span of the differentials of $\iota$ at $x_0$, which is a vector space of dimension at most $N$. This implies that $Q \leq N$, as desired. \qed

\section{The sharpness of Theorem \ref{main}}\label{example section}
In this section we prove Theorem \ref{example1}. For ease of notation we denote the dimension of the ambient cube by $N \geq 2$, and instead use $n \in \nats$ as an iterative index.

Our construction requires a preliminary result regarding the capacity of a point. 

\begin{proposition}\label{capacity} Let $\mathcal{S}$ be a rearrangement invariant Banach function space on $\reals^N$ containing a compactly supported function $g \notin L^{N,1}(\reals^N).$ Then for all points $a \in \reals^N$, $\ep>0$ and all $\tau \in [0,1]$, there is a Lipschitz function $\phi \colon \reals^N \to [0,\tau]$ satisfying
\begin{itemize}
\item the support of $\phi$ is a compact subset of $B(a,\ep)$,
\item $\phi$ is constant with value $\tau$ on a neighborhood of the point $a$,
\item $||\Lip \phi||_{\mathcal{S}} \leq \ep.$
\end{itemize}
\end{proposition}

\begin{proof} It suffices to consider the case that $\tau=1$ and $a = 0$.  Let $g \colon \reals^N \to \reals$ be a compactly supported function in $\mathcal{S}\bslash L^{N,1}(\reals^N)$; denote the support of $g$ by $K$.  Then
$$||g||_{L^{N,1}} = \int_{0}^{\mathcal{L}^N(K)} t^{\frac{1}{N}-1} g^*(t) \ dt = \infty.$$
Define $u \colon [0,\mathcal{L}^N(K)] \to \reals$ by
$$u(r) = \int_{r}^{\mathcal{L}^N(K)} t^{\frac{1}{N}-1} g^*(t) \ dt.$$
Since $g \in \mathcal{S}$, the axioms of a Banach function space imply that it is finite $\mu$-almost everywhere. Hence $g^*(r)<\infty$ for each $r>0$, and so the non-increasing property of $g^*$ implies that
\begin{equation}\label{u infinity} u(0)-u(r) = \infty \end{equation}
for every $r>0$.

Fix $\ep>0$, and denote the volume of the unit ball in $\reals^N$ by $\Omega_N$.  For $0<\del<\ep$, define $\phi_{\ep,\del}\colon B(0,\ep) \to [0,1]$ by
$$\phi_{\ep,\del}(x) = \begin{cases}
								1 &  0 \leq |x| \leq \del,\\
								\lambda u(\Omega_N|x|^N) - \Lambda & \del \leq |x| \leq \ep/2,\\
								0 & |x|>\ep/2,\\
								\end{cases}$$
where $\lambda, \Lambda \geq 0$ are chosen to make $\phi_{\ep,\del}$ continuous, i.e.,
$$\lambda = (u(\Omega_N\del^N)-u(\Omega_N(\ep/2)^N))\inv, \ \text{and}\ \Lambda = \lambda u(\Omega_N(\ep/2)^N).$$
A calculation shows that $\phi_{\ep,\del}$ is Lipschitz, and that there is a constant $C>0$ depending only on $N$ such that for almost every $x \in \reals^N$,
$$\Lip \phi_{\ep,\del}(x) = \begin{cases} 	
									C\lambda g^*(\Omega_N|x|^N) & \del\leq |x|\leq \ep/2,\\
									0 & \text{otherwise.}\\
									\end{cases}
$$						
Another calculation shows that the functions $x \in \reals^N \mapsto g^*(\Omega_N|x|^N)$ and $t \in [0,\infty) \mapsto g^*(t)$ have the same distribution functions. Since $\mathcal{S}$ is rearrangement invariant, it follows that the function $x \in \reals^N \mapsto g^*(\Omega_N|x|^N)$ has finite $\mathcal{S}$-norm, and so
$$||\Lip \phi_{\ep,\del}||_{\mathcal{S}} = C\lambda ||g^*(\Omega_N|\cdot|^N)||_{\mathcal{S}}.$$
By \eqref{u infinity}, the quantity $\lambda$ tends to $0$ as $\del$ tends to $0$.  Thus, for sufficiently small $\del$, the function $\phi_{\ep,\del}$ satisfies the stated requirements.		
\end{proof}

We also require one elementary lemma for the proof of Theorem~\ref{example1}.

\begin{lemma}\label{tech lemma} There exists a sequence $\{a_n\}_{n \in \nats}$ of positive real numbers such that
$$\prod_{n=0}^\infty\left((1-a_n)^N-(2a_n)^N\right) > 0.$$
\end{lemma}

\begin{proof} Noting that
$$\lim_{x \to 0^+} (1-x)^N - (2x)^N = 1$$
we may choose the sequence $\{a_n\}$ so that for each $n \in \nats$
$$(1-a_n)^N - (2a_n)^N > e^{-2^{-n}}.$$
The desired inequality now follows from the calculation that
\begin{align*}\log\left(\prod_{n=0}^\infty\left((1-a_n)^N-(2a_n)^N\right)\right) &= \sum_{n=0}^{\infty} \log\left((1-a_n)^N-(2a_n)^N\right) \\ &>-\sum_{n=0}^{\infty}2^{-n} > -\infty.\end{align*}
\end{proof}

\begin{proof}[Proof of Theorem~\ref{example1}]

Let $\{j_n\}_{n \in \nats}$ be a sequence of positive integers satisfying
$$j_0=0, \ j_n\equiv 0\ \text{mod}\ {3}, \ \text{and}\ j_{n+1} \geq 9(j_n +1).$$
Set
$$k_n = \frac{2j_n}{3} \ \text{and}\ l_n = \frac{j_{n+1}}{3} + \frac{2j_n}{3} +1.$$
By choosing the sequence $\{j_n\}$ to grow fast enough, we may assume that for each $n \in \nats$,
\begin{equation}\label{growth}
2^{j_n-l_n} = 2^{-\frac{1}{3}(j_{n+1}-j_n)-1} \leq a_n,
\end{equation}
where $\{a_n\}_{n \in \nats}$ is the sequence from Lemma \ref{tech lemma}.

We use these parameters to define a sequence of collections of dyadic cubes in $[-1,1]^N$.  Let
$$Q_0 = [-1,1]^N \ \text{and}\ \mathcal{Q}_0 = \{Q_0\}.$$
Fix $n \in \nats$, and assume that the collection of $\iQ_n$ of cubes of side-length $2\cdot 2^{-j_n}$ has been defined. Given $Q \in \iQ_n$ with center denoted by $a$, consider the standard subdivision of $Q$ into essentially disjoint cubes of side-length $2 \cdot 2^{-j_{n+1}}.$ We declare such a sub-cube to be an element of the collection $\iQ(Q)$ if its center $a'$ satisfies
$$2(2^{-l_n}) + 2^{-j_{n+1}} \leq ||a-a'||_{\infty} \leq 2^{-j_n} - 2^{-l_n} - 2^{-j_{n+1}}.$$
The set $\iQ(Q)$ is non-empty if the difference between the upper and lower bounds above is at least $2^{1-j_{n+1}}$. A calculation using the estimate $j_{n+1} \geq 9(j_n+1)$ shows that this is always the case.  We inductively define
$$\iQ_{n+1} = \{\iQ(Q) : Q \in \iQ_{n}\}.$$

Given a cube $Q \in \iQ_n$, we denote the concentric sub-cube of side-length $2\cdot 2^{-l_n}$ by $I_Q$. Note that for any $Q \in \iQ_{n}$, the cube $I_Q$ is disjoint from each cube in $\iQ_{n+1}$. Hence, given positive integers $n$ and $m$ and cubes $Q \in \iQ_{n}$ and $Q' \in \iQ_m$, the set $I_Q \cap I_{Q'}$ is empty unless $Q=Q'$. See Figure \ref{con}.

\begin{figure}[h]
\begin{center}
\begin{tikzpicture}
\coordinate (bl) at (0,0);
\coordinate (br) at (8.0,0);
\coordinate (tl) at (0,8.0);
\coordinate (tr) at (8.0,8.0);
\draw (bl)--(br)--(tr)--(tl)--(bl);

\foreach \x in {0,...,28}{
    \draw (.4+\x*.1,.4)--(.4+\x*.1,7.6);
    \draw (4.8+\x*.1,.4)--(4.8+\x*.1,7.6);
    \draw (.4,.4+\x*.1)--(7.6,.4+\x*.1);
    \draw (.4,4.8+\x*.1)--(7.6,4.8+\x*.1);
}

\foreach \x in {1,...,15}{
    \draw (3.2+\x*.1,.4)--(3.2+\x*.1,3.2);
    \draw (3.2+\x*.1,4.8)--(3.2+\x*.1,7.6);
    \draw (.4,3.2+\x*.1)--(3.2,3.2+\x*.1);
    \draw (4.8,3.2+\x*.1)--(7.6,3.2+\x*.1);
}

\fill [black!25] (3.6,3.6) rectangle (4.4,4.4);

\draw [->] (.8,-.6)--(1.3,-.1);
\node at (.8,-.6) [left] {$Q$};
\draw [->] (8.6,5.8)--(4.4,4.4);
\node at (8.6,5.8) [right] {$I_Q$};

\draw [->] (3.6,3.6)--(4.2,3.9);
\node at (3.65,3.65) [below] {$2^{-l_n}$};

\draw [->] (3.6,3.6)--(3.4,3.9);

\draw [|-|] (4,-.4)--(8,-.4);
\node at (6,-.4) [below] {$2^{-j_n}$};

\draw [|-|] (4,4)--(4.4,4);
\node at (4.2,4) [below] {};

\draw [|-|] (3.21,4)--(3.6,4);

\draw [|-|] (8.4,0)--(8.4,.4);
\node at (8.4, .2) [right] {$2^{-l_n}$};
\end{tikzpicture}
\end{center}
\caption{A cube $Q\in \mathcal{Q}_n$, shown with the collection of subcubes $\mathcal{Q}(Q)$}
\label{con}
\end{figure}
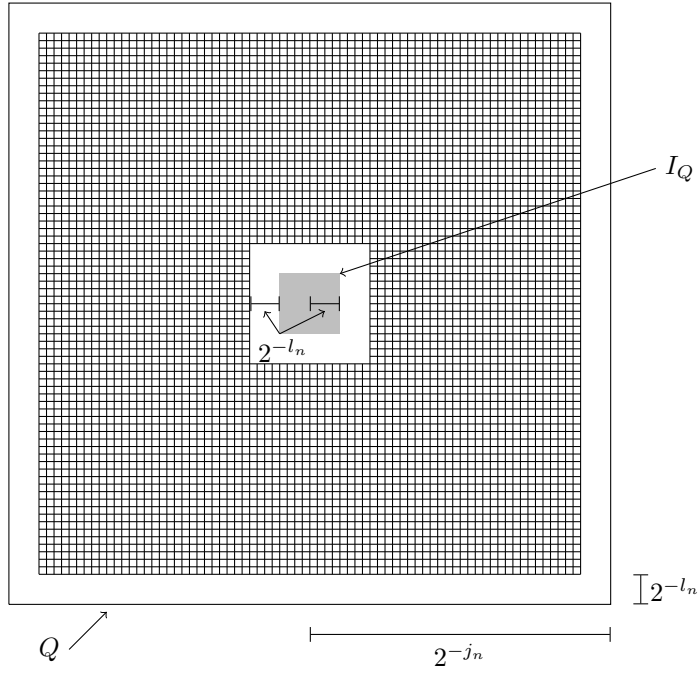

For each $n \in \nats$, fix a number $\ep_n>0$ so small that
$$(\card{\iQ_n})\cdot \ep_n \leq 2^{-n}.$$
By Proposition \ref{capacity}, for each $Q \in \mathcal{Q}_n$, there is a Lipschitz function $$\phi_Q \colon \reals^N \to [0,2^{-k_n}]$$ with the following properties:
\begin{itemize}
\item $\phi_Q \equiv 2^{-k_n}$ on a neighborhood of the center of $I_Q$,
\item $\phi_Q \equiv 0$ on a neighborhood of the boundary of $I_Q$,
\item $||\lip \phi_Q||_{\mathcal{S}} \leq \ep_n.$
\end{itemize}

We define a function $f \colon \reals^N \to \reals$ by declaring that for each $n \in \nats$ and $Q \in \iQ_n$,
$$f|_{I_Q} = \phi_Q,$$
and setting $f \equiv 0$ on the remaining subset of $\reals^N.$

We first claim that $\lip f(x) <\infty$ for every $x \in \reals^N.$ Set
\begin{align*} \mathcal{F} &= \left\{x \in \reals^N : \text{there is a largest $n \in \nats$ such that $x\in \bigcup_{Q \in \iQ_n}Q$}\right\}, \\
\mathcal{I} &=\left\{x \in \reals^N : x\in \bigcup_{Q \in \iQ_n}Q^0 \ \text{for every $n \in \nats$}\right\}, \end{align*}
where $Q^0$ denotes the interior of the cube $Q$. Note that these sets partition $\reals^N$ by construction.

If $x \in \mathcal{F}$ then either $f$ is identically zero in a neighborhood of $x$ or $f \equiv \phi_{Q'}$ on a neighborhood of $x$, for some $Q' \in \iQ_m$, $m \in \nats$. In either case, $\lip f(x) < \infty$.

If $x \in \mathcal{I}$, and $Q \in \iQ_{n}$ contains $x$, then
$$\dist \left(x,\bigcup_{i=0}^n \bigcup_{Q \in \iQ_i} I_Q\right) \geq 2^{-l_n}.$$
Thus, if $y \in \reals^N$ satisfies $|x-y| < 2^{-l_n}$, then $|f(x)-f(y)| \leq 2^{-k_{n+1}}.$ This implies that
$$\lip f(x) \leq \liminf_{n \to \infty} \frac{2^{-k_{n+1}}}{2^{-l_n}} = 0.$$

From the above discussion, we see that for all $x \in \reals^N$,
$$\lip f(x) = \sum_{n=0}^\infty \sum_{Q \in \iQ_n} \lip \phi_Q(x),$$
and thus, by the choice of the sequence $\{\ep_n\}_{n \in \nats}$,
$$||\lip f||_{\mathcal{S}} \leq 2.$$

We now show that $f$ fails to be differentiable at any point $x \in \mathcal{I}$. For each $n \in \nats$, we may find a cube $Q \in \iQ_n$ containing $x$ in its interior; denote the center of $Q$ by $a$. Let $b$ be the point of intersection of the boundary of $Q$ and the ray emanating from $a$ passing through $x$. Then
$$\frac{||f(a)-f(b)||}{||a-b||} \geq \frac{2^{-k_n}}{\sqrt{N}\cdot2^{-j_n}} = \frac{2^{j_n/3}}{\sqrt{N}}.$$
The triangle inequality and the collinearity of $a$, $b$, and $x$ now imply that there is a point $y \in \{a,b\}$ satisfying
$$\frac{||f(x)-f(y)||}{||x-y||} \geq \frac{2^{j_n/3}}{\sqrt{N}}.$$
It follows that
$$\Lip f(x) \geq \limsup_{n \to \infty} \frac{2^{j_n/3}}{2\sqrt{N}} = \infty.$$
This implies that $f$ is not differentiable at $x$.

Finally, we show that $\mathcal{I}$ has positive $N$-dimensional Hausdorff measure. First, note that for any $n \in \nats$ and $Q \in \iQ_n$,
$$\Hdim^N(Q) = (2\cdot 2^{-j_n})^N \ \text{and}\ \Hdim^N\left(\bigcup_{Q' \in \iQ(Q)}Q'\right) = (2\cdot 2^{-j_n}-2\cdot2^{-l_n})^N - (4\cdot2^{-l_n})^N.$$
Thus
$$\frac{\Hdim^N\left(\bigcup_{Q' \in \iQ(Q)}Q'\right)}{\Hdim^N(Q)} = (1-2^{j_n-l_n})^N - (2\cdot 2^{j_n-l_n})^N.$$
Note that
\begin{equation*}
\Hdim^N(\mathcal{I})=\lim_{n\to \infty} \Hdim^N(\cup_{Q'\in \iQ_n} Q').
\end{equation*}
Further, for $n \geq 1$,
\begin{equation*}
\begin{split}
\Hdim^N\left(\bigcup_{Q'\in \iQ_n} Q'\right)&=\Hdim^N\left(\bigcup_{Q\in \iQ_{n-1}}\bigcup_{Q'\in \iQ(Q)} Q'\right)=\sum_{Q\in \iQ_{n-1}}\Hdim^N\left(\bigcup_{Q'\in \iQ(Q)}Q'\right)\\
&=\sum_{Q\in \iQ_{n-1}}\frac{\Hdim^N\left(\bigcup_{Q'\in \iQ(Q)} Q'\right)}{\Hdim^N(Q)}\cdot \Hdim^N(Q)\\
&=
\left((1-2^{j_{n-1}-l_{n-1}})^N-(2\cdot 2^{j_{n-1}-l_{n-1}})^N\right)\Hdim^N\left(\bigcup_{Q\in \iQ_{n-1}}Q\right).
\end{split}
\end{equation*}

Thus
$$\Hdim^N(\mathcal{I})=\Hdim^N((-1,1)^N) \cdot \prod_{n=0}^\infty (1-2^{j_n-l_n})^N - (2\cdot 2^{j_n-l_n})^N.$$
The inequality \eqref{growth} and Lemma \ref{tech lemma} now complete the proof.\end{proof}

\bibliographystyle{alpha}
\bibliography{StepanovRNP}

\end{document}